\documentclass[a4paper,12pt]{article}
\usepackage{amsmath,amsfonts,amsthm} 
\usepackage{graphicx}
\usepackage{cite}

\newcommand{\C}{{\mathbb C}}

\newcommand{\Z}{{\mathbb Z}}
\newcommand{\eps}{\varepsilon}

\newcommand{\ba}{{\mathbf a}}
\newcommand{\bb}{{\mathbf b}}
\newcommand{\bc}{{\mathbf c}}

\newcommand{\cA}{{\mathbf A}}
\newcommand{\cB}{{\mathbf B}}
\newcommand{\cC}{{\mathbf C}}

\newcommand{\bi}{{\mathbf i}}

\def\Mat{{\rm Mat}}
\newcommand{\G}{\Gamma}
\newcommand{\de}{\em}

\newtheorem{theorem}{Theorem}[section]
\newtheorem{lemma}[theorem]{Lemma}
\newtheorem{corollary}[theorem]{Corollary}

\theoremstyle{definition}
\newtheorem{note}[theorem]{Note}
\newtheorem{definition}[theorem]{Definition}

\newtheorem{rem}[theorem]{Remark}

\begin{document}
\pagestyle{plain}

\title{Distance-regular graphs of $q$-Racah type\\
       and the universal Askey-Wilson algebra}
\author{Paul Terwilliger and Arjana \v{Z}itnik}
%\date{July 17, 2013}
\date{}
\maketitle

% %- - - - - - - - - - - - - - ABSTRACT Abstract - - - - - - - - - - - -
\begin{abstract} \noindent
%zz
Let $\C$ denote the field of complex numbers,
and fix a nonzero $q \in \C$ 
such that $q^4 \ne 1$. Define a  $\C$-algebra $\Delta_q$  by  generators 
and relations in the following way. The  generators are $A,B,C$. 
The relations assert that each of
$$
A+\frac{qBC-q^{-1}CB}{q^2-q^{-2}}, \ \ \
B+\frac{qCA-q^{-1}AC}{q^2-q^{-2}}, \ \ \
C+\frac{qAB-q^{-1}BA}{q^2-q^{-2}}
$$
is central in $\Delta_q$. The algebra $\Delta_q$ is called the 
universal Askey-Wilson algebra.
Let $\Gamma$ denote a distance-regular graph that has $q$-Racah type.
Fix a vertex $x$ of $\Gamma$ and let $T=T(x)$ denote
the corresponding subconstituent algebra.
In this paper we discuss a relationship between $\Delta_q$
and
$T$.
Assuming that every irreducible $T$-module is thin,
we display a surjective $\C$-algebra homomorphism $\Delta_q \to T$. 
This gives a $\Delta_q$ action on the standard module of $T$.

\bigskip



\end{abstract}

%--------------------------------------------------------------------------
%--------------------sec:introduction--1se-----------------------------
\section{Introduction}

The universal Askey-Wilson algebra  $\Delta_q$ was introduced 
by the first author in \cite{TerwilligerUAWA}.
The algebra $\Delta_q$ is a central extension of the Askey-Wilson algebra  
\cite{ZhedanovAW}. The algebra $\Delta_q$ is related to
the $q$-Onsager algebra  \cite[Section~9]{TerwilligerUAWA},
the algebra $U_q(\mathfrak{sl}_2)$ \cite{TerwilligerUAWAEq},
and the double affine Hecke algebra of type $(C_1^{\vee}, C_1)$ 
\cite{TerwilligerUAWADAHA}. 
In \cite{Huang1} the finite-dimensional irreducible $\Delta_q$-modules 
are classified up to isomorphism, under the assumption that $q$
is not a root of unity.
In this paper we describe how $\Delta_q$ is related to the subconstituent
algebra of a distance-regular graph
that has $q$-Racah type. 

Before we describe our results we set some  conventions.
An algebra is meant to be associative and have
a 1. A subalgebra has the same 1 as the parent algebra.
%Throughout the paper $\C$ denotes the field of complex numbers.
The field of complex numbers is denoted $\C$.
Until the end of Section 5 we fix a nonzero $q \in \C$ such that $q^4 \ne 1$.

We now recall the universal Askey-Wilson algebra
\cite[Definition~1.2]{TerwilligerUAWA}. 
Define a $\C$-algebra $\Delta_q$ by  generators 
and relations in the following way. The  generators are $A,B,C$. 
The relations assert that each of
\begin{equation}                              \label{def:AWalgebra}
A+\frac{qBC-q^{-1}CB}{q^2-q^{-2}}, \ \ \
B+\frac{qCA-q^{-1}AC}{q^2-q^{-2}}, \ \ \
C+\frac{qAB-q^{-1}BA}{q^2-q^{-2}}
\end{equation}
is central in $\Delta_q$. The algebra $\Delta_q$ is called the 
{\de universal Askey-Wilson algebra}.

We now recall some notions concerning distance-regular graphs
(see Sections 3 and 4 for formal definitions).
Let $\G$ denote a distance-regular graph with diameter $D \ge 3$.
We say that $\G$ has $q$-Racah type whenever $\G$ has a 
$Q$-polynomial structure with eigenvalue sequence $\{\theta_i\}_{i=0}^D$
and dual eigenvalue sequence $\{\theta_i^*\}_{i=0}^D$ such that
\begin{eqnarray*}
\theta_i   = w + u  q^{2i-D} + v  q^{D-2i},
\qquad \qquad 
\theta_i^* = w^* + u^*  q^{2i-D} + v^*  q^{D-2i}
\end{eqnarray*}
for $0 \leq i \leq D$, with $u, u^*, v, v^*$ nonzero. 
Assume that $\G$ has $q$-Racah type. 
Fix a vertex $x$ of $\G$ and let $T = T(x)$ denote the corresponding 
subconstituent algebra. The algebra $T$ is generated by
the adjacency matrix $A$ and the dual adjacency matrix $A^*=A^*(x)$.
The algebra $T$ is semisimple \cite[Lemma 3.4]{Talg}.
An irreducible $T$-module is said to be thin
whenever its intersection 
with each eigenspace of $A$ and each eigenspace of $A^*$ has dimension at
most 1. Assume that every irreducible $T$-module is thin.
In our main result, 
we display a surjective $\C$-algebra homomorphism $\Delta_q \to T$. 
This gives a $\Delta_q$ action on the standard module of $T$. 
%Assuming that every irreducible $T$-module is thin,
%we display a surjective $\C$-algebra homomorphism $\Delta \to T$. 
%This gives a $\Delta$ action on the standard module of $\G$.
%
To obtain our main results we invoke the theory of Leonard pairs
and Leonard systems \cite{Huang}, \cite{TerwilligerTD}, \cite{TerwilligerMadrid}.

The paper is organized as follows. In Section 2 we recall some
background concerning Leonard pairs and Leonard systems.
In Section 3 we  recall the basic theory concerning 
a distance-regular graph and its subconstituent algebras.
In Section 4 we consider the irreducible modules 
for a subconstituent algebra.
In Section 5 we describe our main results. 
In Section 6 we apply our results to the 2-homogeneous 
bipartite distance-regular graphs.

%\texttt{ Should it  be described in more detail how the main results
%were obtained? Then we also have to recall some notions regarding the 
%Leonard Pairs.}

%--------------------------------------------------------------------------
%----------------------------- 2se-----------------------------------
\section{Leonard pairs and Leonard systems}
\label{sec:LP}

In this section we recall some background concerning Leonard pairs 
and Leonard systems. For more information we refer the reader to
\cite{Huang}, \cite{TerwilligerTD}, \cite{TerwilligerMadrid}. 
Throughout this section fix an integer $d \ge 0$.
Let $V$ denote a vector space over $\C$ with  dimension $d+1$. 
Let $\mbox{End}(V)$ denote the $\C$-algebra of all $\C$-linear
transformations from $V$ to $V$.

We recall the definition of a Leonard pair. We  use the following terms.
Let $S$ denote a square matrix with entries in $\C$.
Then $S$ is called {\de tridiagonal} whenever each nonzero entry lies
on either the diagonal, the subdiagonal, or the superdiagonal. 
Assume that $S$ is tridiagonal. Then $S$ is called {\de irreducible} 
whenever each entry on the subdiagonal is nonzero 
and each entry on the superdiagonal is nonzero.

%--------------------------------------------------------------------------
\begin{definition}   
\label{def:LP}
\rm 
\cite[Definition~1.1]{TerwilligerTD}.
By a {\em Leonard pair} on $V$, we mean an ordered pair $A,A^*$ of elements 
in $\mbox{End}(V)$ that satisfy the conditions (i), (ii) below.
\begin{itemize}
\item[(i)] There exists a basis for $V$ with respect to which the matrix 
           representing $A$ is irreducible tridiagonal and
           the matrix representing $A^*$ is diagonal.
\item[(ii)] There exists a basis for $V$ with respect to which the matrix 
           representing $A^*$ is irreducible tridiagonal and
           the matrix representing $A$ is diagonal.
\end{itemize}
\end{definition}
%--------------------------------------------------------------------------

A Leonard pair is  closely related  to an object called a
Leonard system \cite{TerwilligerTD}.
To define a Leonard system we use the following terms.
Pick $A \in \mbox{End}(V)$. The map $A$  is said to be 
{\it diagonalizable} whenever $V$ is spanned by
the eigenspaces of $A$.
For the moment assume that $A$ is diagonalizable, and let $W$
denote an eigenspace of $A$. Then there exists a unique $E \in
\mbox{End}(V)$
such that $E-I$ is zero on $W$ and $E$ is zero on all the
other eigenspaces of $A$. The map $E$ is called the
{\it primitive idempotent} of $A$ that corresponds to $W$.
The map $A$ is said to be {\de multiplicity-free} 
whenever $A$ has $d+1$ distinct eigenvalues. 
If $A$ is multiplicity-free then $A$ is diagonalizable.

%--------------------------------------------------------------------------
\begin{definition}                            \label{def:LS}
\rm 
\cite[Definition~1.4]{TerwilligerTD}.
By a {\em Leonard system} on $V$ we mean a sequence 
$\Phi=(A,\{E_i\}_{i=0}^d,A^*,\{E_i^*\}_{i=0}^d)$ that satisfies (i)--(v) below.
\begin{itemize}
\item[(i)] Each of $A,A^*$ is a multiplicity-free element in $\mbox{End}(V)$.
\item[(ii)] $\{E_i\}_{i=0}^d$ is an ordering of the primitive idempotents of $A$.
\item[(iii)] $\{E_i^*\}_{i=0}^d$ is an ordering of the primitive idempotents of $A^*$.
\item[(iv)] $\displaystyle{E_iA^*E_j=\left\{
                \begin{array}{ll}
                   0 & \mbox{if} \ \ |i-j|>1\\
                   \ne 0 & \mbox{if} \ \ |i-j|=1
                 \end{array} \right.   
               \ \ \ \ (0 \le i,j \le d).}$
\item[(v)] $\displaystyle{E_i^*AE_j^*=\left\{
                \begin{array}{ll}
                   0 & \mbox{if} \ \ |i-j|>1\\
                   \ne 0 & \mbox{if} \ \ |i-j|=1
                 \end{array} \right.   
               \ \ \ \ (0 \le i,j \le d).}$
\end{itemize}
We call $d$ the {\em diameter} of $\Phi$.
\end{definition}
%--------------------------------------------------------------------------

Leonard pairs and Leonard systems are related as follows.
Let $\Phi=(A,\{E_i\}_{i=0}^d,A^*,$ $\{E_i^*\}_{i=0}^d)$ denote a 
Leonard system on $V$.
For $0 \le i \le d$ let $v_i$ denote a nonzero vector in $E_iV$. 
Then the sequence $\{v_i\}_{i=0}^{d}$ is a basis for $V$ that satisfies
Definition \ref{def:LP}(ii).
For $0 \le i \le d$ let $v_i^*$ denote a nonzero vector in $E_i^*V$. 
Then the sequence $\{v_i^*\}_{i=0}^{d}$ is a basis for $V$ that satisfies
Definition \ref{def:LP}(i). Therefore the pair $A,A^*$ is a Leonard pair on $V$.
Conversely, let $A,A^*$ denote a Leonard pair on $V$.
By \cite[Lemma 1.3]{TerwilligerTD}, each of $A$, $A^*$ is multiplicity-free.
Let $\{v_i\}_{i=0}^{d}$ denote a basis for $V$ that satisfies Definition \ref{def:LP}(ii).
For $0 \le i \le d$ the vector $v_i$ is an eigenvector for $A$; 
let  $E_i$ denote the corresponding primitive idempotent. 
Let $\{v_i^*\}_{i=0}^{d}$ denote a basis for $V$ that satisfies Definition \ref{def:LP}(i).
For $0 \le i \le d$ the vector $v_i^*$ is an eigenvector for $A^*$; 
let  $E_i^*$ denote the corresponding primitive idempotent. 
Then $\Phi=(A,\{E_i\}_{i=0}^d,A^*,\{E_i^*\}_{i=0}^d)$ is a Leonard system on $V$.
We say that the Leonard pair $A,A^*$ and the Leonard system $\Phi$ are {\de associated}.
%Observe that each Leonard system is associated with a unique Leonard pair.

%--------------------------------------------------------------------------
\begin{definition}
Let $\Phi=(A,\{E_i\}_{i=0}^d,A^*,\{E_i^*\}_{i=0}^d)$ denote a Leonard system on $V$.
For $0 \le i \le d$ let $\theta_i$ (resp. $\theta_i^*$) 
denote the eigenvalue of $A$ (resp. $A^*$)
associated with $E_i$ (resp. $E_i^*$).
We call $\{\theta_i\}_{i=0}^d$ (resp. $\{\theta_i^*\}_{i=0}^d$)
the {\de eigenvalue sequence} (resp. {\de dual eigenvalue sequence} )
of  $\Phi$. 
\end{definition}

Let $\Phi=(A,\{E_i\}_{i=0}^d,A^*,\{E_i^*\}_{i=0}^d)$ denote a Leonard system.
We just defined the eigenvalue sequence and the dual eigenvalue sequence of $\Phi$.
Associated with $\Phi$  are two more sequences of 
scalars, called the first split sequence and
second split sequence \cite[p. 155]{TerwilligerTD}.
These sequences are denoted by $\{\varphi_i\}_{i=1}^d$ and $\{\phi_i\}_{i=1}^d$, respectively.
The sequence 
$(\{\theta_i\}_{i=0}^d,\{\theta_i^*\}_{i=0}^d,
  \{\varphi_i\}_{i=1}^d,\{\phi_i\}_{i=1}^d)$ is called the parameter array
  of $\Phi$.
The Leonard system $\Phi$ is uniquely determined by its parameter array,  up to 
isomorphism of Leonard systems \cite[Theorem 1.9]{TerwilligerTD}.
By \cite[Lemma 5.2, Lemma 6.4]{TerwilligerTD} both
\begin{eqnarray}
\varphi_i &=& (\theta_{i-1}^*-\theta_i^*) \sum_{h=0}^{i-1}\, (a_h-\theta_h),
    \label{eq:calculatevarphi}\\
\phi_i &=& (\theta_{i-1}^*-\theta_i^*) \sum_{h=0}^{i-1}\, (a_h-\theta_{d-h})
    \label{eq:calculatphi}
\end{eqnarray}    
for $1 \le i \le d$, where 
\begin{equation}                   \label{eq:calculateah}
a_h ={\rm trace}(E_h^*A)  \ \ \ \ \ \ \ \ (0 \le h \le d).
\end{equation}.

%--------------------------------------------------------------------------
\begin{definition}            \label{def:LeonardQRacah}
Let $\Phi=(A,\{E_i\}_{i=0}^d,A^*,\{E_i^*\}_{i=0}^d)$ denote 
a Leonard system  on $V$, with  eigenvalue sequence $\{\theta_i\}_{i=0}^d$ 
and dual eigenvalue sequence $\{\theta_i^*\}_{i=0}^d$. 
We say that $\Phi$ has {\de $q$-Racah type} whenever there exist nonzero 
$a,b$ in $\C$ such that  both
\begin{equation*}
\theta_i = aq^{2i-d} + a^{-1}q^{d-2i}, \ \ \ \ \ \ \ \ \ \ 
\theta_i^* = bq^{2i-d} + b^{-1}q^{d-2i} 
\end{equation*}
for $0 \le i \le d$.
\end{definition}
%--------------------------------------------------------------------------

%--------------------------------------------------------------------------
\begin{definition}            \label{def:LPQRacah}
A Leonard pair is said to have {\de $q$-Racah type} whenever an associated
Leonard system has $q$-Racah type.
\end{definition}
%--------------------------------------------------------------------------

Let $\Phi=(A,\{E_i\}_{i=0}^d,A^*,\{E_i^*\}_{i=0}^d)$
denote a Leonard system with parameter array 
$(\{\theta_i\}_{i=0}^d,$ $\{\theta_i^*\}_{i=0}^d,
  \{\varphi_i\}_{i=1}^d,\{\phi_i\}_{i=1}^d)$.
Assume that $\Phi$ has $q$-Racah type.
%with parameters $a,b$ 
%as in Definition \ref{def:LeonardQRacah}.
Let $c \in \C$ denote a root of the equation
\begin{equation}                   \label{eq:kappa}
\xi^2-\kappa \xi +1=0,
\end{equation}
where $\kappa=0$ for $d=0$, and for $d \geq 1$,
\begin{equation}              \label{eq:calculatekappa}          
\kappa = ab^{-1}q^{d-1}+a^{-1}bq^{1-d} + 
\frac{\phi_1}{(q-q^{-1})(q^d-q^{-d})}
\end{equation}
with $a,b$  from
 Definition \ref{def:LeonardQRacah}.
From the form of 
(\ref{eq:kappa}) we see that
 $c \ne 0$. Moreover $c^{-1}$ is also a root of (\ref{eq:kappa}), 
and (\ref{eq:kappa}) has no further roots.
By \cite[Lemma 6.3]{Huang}, both
\begin{eqnarray*}
\varphi_i &=& a^{-1}b^{-1}q^{d+1}(q^i-q^{-i})(q^{i-d-1}-q^{d-i+1})
              (q^{-i}-abcq^{i-d-1})(q^{-i}-abc^{-1}q^{i-d-1}),\\
\phi_i &=& ab^{-1}q^{d+1}(q^i-q^{-i})(q^{i-d-1}-q^{d-i+1})
              (q^{-i}-a^{-1}bcq^{i-d-1})(q^{-i}-a^{-1}bc^{-1}q^{i-d-1})
\end{eqnarray*}
for $1 \le i \le d$. 

We  now recall the Askey-Wilson relations
\cite{ZhedanovAW}.
 We will work with the
$\Z_3$-symmetric version
\cite[Theorem 10.1]{Huang}.

\begin{theorem}{\rm \cite[Theorem 10.1]{Huang}.}      \label{thm:AskeyWilsonZ3}
Let $A,A^*$ denote a Leonard pair on $V$ that has $q$-Racah type.
Recall the scalars $a,b$ from Definition {\rm \ref{def:LeonardQRacah}}
and $c$ from below Definition {\rm \ref{def:LPQRacah}}.
Then there exists a unique $A^\eps \in {\rm End}(V)$ such that
\begin{eqnarray}
A+\frac{qA^*A^\eps-q^{-1}A^\eps A^*}{q^2-q^{-2}} &=&
     \frac{(a+a^{-1})(q^{d+1}+q^{-d-1})+(b+b^{-1})(c+c^{-1})}{q+q^{-1}}I,  \label{eq:AWsim1}\\
A^*+\frac{qA^\eps A -q^{-1}AA^\eps}{q^2-q^{-2}} &=&
     \frac{(b+b^{-1})(q^{d+1}+q^{-d-1})+(c+c^{-1})(a+a^{-1})}{q+q^{-1}}I, \label{eq:AWsim2}\\
A^\eps+\frac{qAA^*-q^{-1}A^*A}{q^2-q^{-2}} &=&
     \frac{(c+c^{-1})(q^{d+1}+q^{-d-1})+(a+a^{-1})(b+b^{-1})}{q+q^{-1}}I. \label{eq:AWsim3}
\end{eqnarray}
\end{theorem}
%--------------------------------------------------------------------------

%--------------------------------------------------------------------------
%--------------------sec:preliminaries-3se-----------------------------
\section{Background on distance-regular graphs}
\label{sec:prelim}

In this section we review some  basic concepts 
concerning distance-regular graphs.
See Brouwer, Cohen and Neumaier \cite{BCN} and Terwilliger 
\cite{Talg,Talg2,Talg3}
for more background information.

Let $X$ denote a nonempty finite set.
Let  $\Mat_X(\C)$ denote the $\C$-algebra consisting of the matrices 
whose rows and columns are indexed by $X$ and whose entries are in $\C$. 
Let $V = \C^X$ denote the vector space over $\C$ consisting of the column 
vectors whose coordinates are indexed by $X$ and whose entries are in $\C$.
Observe that $\Mat_X(\C)$ acts on $V$ by left multiplication.
We endow $V$ with the Hermitean inner product 
$\langle \, , \, \rangle$ that satisfies
$\langle u, v \rangle = u^t \bar{v}$ for $u, v \in V$, 
where $t$ denotes transpose and $\bar{ }$  denotes complex conjugation.
%We call $V$ the {\em standard module}. % of $\Mat_X(\C)$.
%For all $y \in X$, let $\hat{y}$ denote the element of 
%$V$ with $1$ in the $y$-th coordinate and $0$ in all other coordinates. 
%
Let $W,W'$ denote 
nonempty subsets of $V$.
These subsets are said
to be {\em orthogonal} whenever $\langle w,w^\prime \rangle=0$ for all 
$w \in W$ and $w^\prime \in W^\prime$.

Let $\G$ denote a finite, undirected, connected
graph, without loops or multiple edges, with vertex set $X$, 
path-length distance function $\partial$ and
diameter %$D=\max \{\partial(x,y)\, |\, x,y \in X\}$.
$D=\max \{\partial(x,y)\, |\,$ $x,y \in X\}$.
For a vertex $x \in X$ and integer  $i \ge 0$
define $$\G_i(x)= \{ y \in X \, \vert \ \partial (x,y)=i \}.$$
%For notational convenience abbreviate $\G(x)=\G_1(x)$.
%For an integer $k \ge 0$, the graph $\G$ is said to be 
%{\it regular with valency k} whenever $|\G(x)|= k$ for all $x\in X$.
The graph $\G$ is said to be {\de distance-regular} 
whenever for all integers $h,i,j$ $(0 \le h,i,j\le D)$ 
and vertices $x,y \in X$ with $\partial(x,y)=h$,
the number $p_{ij}^h=|\G_i(x)\cap \G_j(y)|$ is independent of $x,y$.
The constants $p_{ij}^h$ are called the {\de intersection numbers} of $\G$.
{F}rom now on  assume that $\G$ is distance-regular with diameter $D \ge 3$.
%Note that $\G$ is regular with valency $k=p^0_{11}$.

%--------------------Bose-Mesner------------------
We recall the Bose-Mesner algebra of $\G$.
For $0\le i\le D$ let $A_i$ denote the matrix in $\Mat_X(\C)$ with 
$(x,y)$-entry 
$$
(A_i)_{xy}\, =\, \left\{ \begin{array}{ll}
                          1; & \mbox{if} \ \ \partial(x,y)=i\\
                          0; & \mbox{if} \ \ \partial(x,y)\ne i\\
                           \end{array} \right. 
\ \ \ \ \ (x,y \in X).
$$
We call $A_i$ the {\de $i$-th distance matrix} of $\G$. 
Note that $A_i$ is real and symmetric. 
Observe that $A_0=I$, where $I$ is the identity matrix
in 
 $\Mat_X(\C)$.
The matrix $A=A_1$ is the  {\em adjacency matrix} of $\G$.
We observe that  $J=\sum_{i=0}^D A_i$,
where $J \in
 \Mat_X(\C)$ has all entries 1.
Moreover
$A_iA_j=\sum_{h=0}^D p_{ij}^h A_h $ for $0 \le i,j \le D$.
Let $M$ denote the subalgebra of $\Mat_X(\C)$ generated by $A$.
By \cite[p.~127]{BCN} the matrices $\{ A_i\}_{i=0}^D$ form a basis for $M$.
We call $M$ the {\de Bose-Mesner algebra} of $\G$. 
%
% which implies that the matrices in $\{ A_i\}_{i=0}^D$ are linearly 
% independent and all the powers of $A$ can be expressed as a linear 
% combination of the matrices from $\{ A_i\}_{i=0}^D$, 
% cf. \cite[p.~127]{BCN}. Therefore they form a basis for $M$.
%
By \cite[p.~45]{BCN}, $M$ has a basis $\{E_i\}_{i=0}^D$
such that (i) $E_0=|X|^{-1}J$; (ii) $I=\sum_{i=0}^D E_i$;
(iii) $E_iE_j =\delta_{ij}E_i$ for $0 \le i,j\le D$.
By \cite[p.~59,\, 64]{BI} the matrices $\{E_i\}_{i=0}^D$ are real 
and symmetric. We call $\{ E_i\}_{i=0}^D$ the {\de primitive idempotents} 
of $\G$. 
We call $E_0$ {\it trivial}.
%For $0 \le i \le D$ let $m_i$ denote the rank of $E_i$.
Since $\{ E_i\}_{i=0}^D$ form a basis for $M$, there
exist complex scalars $\{ \theta_i\}_{i=0}^D$ such that 
\begin{equation}                   \label{prodAEi0}
A\ =\ \sum_{i=0}^D \theta_i E_i.
\end{equation}
By (\ref{prodAEi0}) and since $E_i E_j = \delta_{ij}E_i$ we have
\begin{equation}                   \label{prodAEi}
AE_i\, =\, E_iA\, =\, \theta_i E_i \ \ \ \   
\mbox{($0\le i\le D$)}.
\end{equation}
We call  $\theta_i$ the {\de eigenvalue} of $\G$
corresponding to $E_i$. Note that the eigenvalues $\{ \theta_i\}_{i=0}^D$ 
are mutually distinct since $A$ generates $M$. 
Moreover $\{\theta_i\}_{i=0}^D$ are real, since $A$ and 
$\{ E_i\}_{i=0}^D$ are real.  The vector space $V$ decomposes as
$$
V=\sum_{i=0}^DE_iV  \ \ \ \ \mbox{(orthogonal direct  sum)}.
$$

%--------------------Q-poly------------------
We now recall the Krein parameters.
Let $\circ$ denote the entry-wise product in $\Mat_X(\C)$.
Observe $A_i \circ A_j =\delta_{ij}A_i$ for $0\le i,j\le D$,
so $M$ is closed under $\circ$.
Therefore there exist complex scalars $q^h_{ij}$ such that
\begin{equation}\label{kreindef}
%E_i \circ E_j = {1\over |X|}\, \sum_{h=0}^D q^h_{ij} E_h
E_i \circ E_j = |X|^{-1}\, \sum_{h=0}^D q^h_{ij} E_h
\ \ \ \ \ \ \ (0 \le i,j\le D).
\end{equation}
We call the $q^h_{ij}$ the {\de Krein parameters} of $\G$.
These parameters are real and nonnegative \cite[p.~48--49]{BCN}.
Let $\{ E_i\}_{i=1}^D$ denote an ordering of the nontrivial primitive 
idempotents of $\G$. 
The ordering $\{ E_i\}_{i=1}^D$ is called {\it $Q$-polynomial} whenever
for all $0\le i,j\le D$ the Krein parameter $q^1_{ij}$ is zero 
if $|i-j|>1$ and nonzero if $|i-j|=1$. The graph $\G$ is called
{\de $Q$-polynomial} (with respect to the given ordering $\{ E_i\}_{i=1}^D$)
whenever the  ordering $\{ E_i\}_{i=1}^D$ is $Q$-polynomial.
Until the end of Section 5 we  assume that $\G$ is $Q$-polynomial 
with respect to $\{ E_i\}_{i=1}^D$.

%--------------------dual Bose-Mesner-----------------
We recall the dual Bose-Mesner algebras of $\G$ \cite[p.~378]{Talg}. 
For the rest of this paper, fix a vertex $x \in X$. 
For $0\le i \le D$ let $E_i^* = E_i^*(x)$ 
denote the diagonal matrix in $\Mat_X(\C)$ with $(y,y)$-entry
$$
(E_i^*)_{yy}\, =\, \left\{ \begin{array}{ll}
                          1; & \mbox{if} \ \ \partial(x,y)=i\\
                          0; & \mbox{if} \ \ \partial(x,y)\ne i\\
                           \end{array} \right. 
\ \ \ \ \ (y \in X).
$$
%(that is, the matrix $E_i^*$ is a diagonal matrix with 
%the characteristic set of $\G_i(x)$ on the diagonal). 
%
We call $E_i^*$ the {\de $i$-th dual idempotent of $\G$ with 
respect to $x$.}
%For convenience, we also define $E_{-1}^*=E_{D+1}^*=0$. 
Observe  that $I=\sum_{i=0}^D E_i^*$
and  $E_i^*E_j^*=\delta_{ij}E_i^*$ for $0\le i,j\le D$.
Therefore the matrices $\{ E_i^*\}_{i=0}^D$ form a basis for 
a commutative subalgebra $M^*=M^*(x)$ of  $\Mat_X(\C)$.
We call $M^*$ the {\de dual Bose-Mesner algebra} of $\G$
with respect to $x$. 

For $0\le i \le D$ let $A_i^*=A_i^*(x)$ denote the diagonal matrix
in $\Mat_X(\C)$ with $(y,y)$-entry
$$
(A_i^*)_{yy}\, =\, |X|\,(E_i)_{xy} \ \ \ \ \ (y \in X).
$$
We call $A_i^*$ the {\it dual distance matrix} 
corresponding to $E_i$ and $x$.
By \cite[p.~379]{Talg} the matrices $\{ A_i^*\}_{i=0}^D$ form a 
basis for $M^*$.  We abbreviate $A^* = A_1^*$ and call this the
{\em dual adjacency matrix} of $\G$ with respect to $x$. 
The matrix $A^*$ generates $M^*$ \cite[Lemma 3.11]{Talg}.

Since $\{ E_i^*\}_{i=0}^D$ form a basis for $M^*$, there
exist complex scalars $\{ \theta_i^*\}_{i=0}^D$ such that 
\begin{equation}                   \label{prodAstarEi0}
A^*\ =\ \sum_{i=0}^D \theta^*_i E_i^*.
\end{equation}
By (\ref{prodAstarEi0}) and since $E_i^* E_j^* = \delta_{ij}E_i^*$ we have
\begin{equation}                   \label{prodAstarEi}
A^*E_i^*\, =\, E_i^*A^*\, =\, \theta_i^* E_i^* \ \ \ \   
\mbox{($0\le i\le D$)}.
\end{equation}
We call $\theta_i^*$ the {\de dual eigenvalue} of $\G$
corresponding to $E_i^*$. Note that the dual eigenvalues $\{ \theta_i^*\}_{i=0}^D$ 
are mutually distinct since $A^*$ generates $M^*$. 
Moreover $\{\theta_i^*\}_{i=0}^D$ are real, since $A^*$ and 
$\{ E_i^*\}_{i=0}^D$ are real. The vector space $V$ decomposes as
$$
V=\sum_{i=0}^DE_i^*V  \ \ \ \ \mbox{(orthogonal direct  sum)}.
$$

%We now recall how $M$ and $M^*$ are related. 
%By the definition of the distance matrices and dual idempotents, we have
%%
%\begin{equation}                                  \label{combcoeff}
%E_h^* A_i E_j^* = 0\ \ \mbox{if and only if $p_{ij}^h=0$} \ \ \ \ \ \  
%(0 \le h,i,j\le D).
%% in particular, E_i^* A E_h^*\, =\, 0\ \ \mbox{if $|i-h|>1.$}
%\end{equation}
%%
%By \cite[Lemma 3.2]{Talg},
%%
%\begin{equation}                                  \label{propertiesdualbma}
%E_h A_i^* E_j\,=\, 0\ \ \mbox{if and only if\ \ $q_{ij}^h=0$}  \ \ \ \ \ \ 
%(0 \le h,i,j\le D).
%\end{equation}

Let $T = T (x)$ denote the subalgebra of $\Mat_X(\C)$ generated by 
$M$ and $M^*$. We call  $T$ the {\de subconstituent algebra} 
(or {\de Terwilliger algebra}) of $\G$ with respect to  $x$ \cite[p.~380]{Talg}.
Note that $T$ is generated by  $A$, $A^*$.
Since each of $A$ and $A^*$ is real and symmetric,
$T$ is closed under the conjugate-transpose map.
Hence $T$ is semisimple  \cite[Lemma 3.4]{Talg}.

%Observe that an ordering $\theta_0, \theta_1,\dots,\theta_D$ of eigenvalues 
%for $A$ is   $Q$-polynomial if 
%\begin{equation}
%A_1^*E_iV=E_{i-1}V+E_iV+E_{i+1}V   \ \ \ \ \ (0 \le i \le D)
%\end{equation}
%with $E_{-1}=E_{D+1}=0)$. In this case $A_1^*$
%is abbreviated $A^*$ and called the {\de dual adjacency matrix}.

%--------------------------------------------------------------------------
%----------------------------- 4se-----------------------------------
\section{$T$-modules}

We continue to discuss the subconstituent algebra $T$
of our distance-regular graph $\Gamma$.
In this section we consider the $T$-modules.
%Let $W$ be a subspace of $V$. 
By a {\de $T$-module} we mean a subspace $W$ of $V$ such that
$SW \subseteq W$ for all $S \in T$.
We refer to $V$ itself as the {\de standard module} for $T$.
A $T$-module $W$ is said to be {\de irreducible} whenever $W$ is nonzero 
and contains no $T$-modules other than $0$ and $W$. 

Let $W$ denote a $T$-module. Since $T$ is closed under the conjugate-transpose
map, the orthogonal complement $W^\perp$ is also a $T$-module.
Therefore $V$ decomposes into an orthogonal direct sum of
irreducible $T$-modules. 

%The following is proven in \cite[Lemma 12.1]{ItoTerwilliger}.
%
%\begin{lemma}{\rm \cite[Lemma 4.2]{ItoTerwilliger1}.}    \label{lemma:action_irreducible}
%For $Y \in \Mat_X(\C)$ the following are equivalent:
%\begin{itemize}
%\item[{\rm (i)}] $Y \in T$;
%\item[{\rm (ii)}] $YW \subseteq W$ for all irreducible $T$-modules $W$.
%\end{itemize}
%\end{lemma}
%
%The vector space $V$ decomposes as
%$$
%V=\sum_{i=1}^DE_iV  \ \ \ \ \ \ \ \ \ \ \mbox{orthogonal direct sum}.
%$$

We now recall the notion of endpoint, dual endpoint, diameter, and
dual diameter. Let $W$ denote an irreducible $T$-module. Observe that 
$W$ is a direct sum of the nonzero spaces among $E_0^*W,\dots, E_D^*W$. 
Similarly, $W$ is a direct sum of the nonzero spaces among $E_0W,\dots, E_DW$. 
Define
$\rho = \min\{i \, \vert \ 0 \le i \le D, E_i^*W \ne 0\}$ and
$\tau = \min\{i \, \vert \ 0 \le i \le D, E_iW \ne 0\}$. 
We call $\rho$ and $\tau$ the {\de endpoint} and {\de dual endpoint} of $W$, 
respectively. Define
$d=|\{i \, \vert \ 0 \le i \le D, E_i^*W \ne 0\}|-1$ and
$d^*=|\{i \, \vert \ 0 \le i \le D, E_iW \ne 0\}|-1$. We call $d$ and $d^*$ the
{\de diameter} and {\de dual diameter} of $W$, respectively.
By \cite[Corollary 3.3]{Pascasio} the diameter of $W$ is equal to the
dual diameter of $W$. 
By \cite[Lemma 3.9, Lemma 3.12]{Talg} $\dim E_i^*W \le 1$ for $0 \le i \le D$ 
if and only if $\dim E_iW \le 1$ for $0 \le i \le D$. 
In this case $W$ is said to be {\de thin}.

%--------------------------------------------------------------------------
\begin{lemma}{\rm \cite[Lemma 3.4, Lemma 3.9, Lemma 3.12]{Talg}.}       \label{thin}
Let $W$ denote an irreducible $T$-module with endpoint $\rho$, 
dual endpoint $\tau$, and diameter $d$. Then $\rho,\tau,d$ are nonnegative
integers such that $\rho+d \le D$ and $\tau + d \le D$. Moreover the following
{\rm (i)--(iv)} hold.
\begin{itemize}
\item[{\rm (i)}] $E_i^*W \ne 0$ if and only if $\rho \le i \le \rho+d,$ \ \ \ $(0 \le i \le D)$.
\item[{\rm (ii)}] $W=\sum_{h=0}^d \, E_{\rho+h}^*W$ \ \ \  \mbox{\rm (orthogonal direct sum)}.
\item[{\rm (iii)}] $E_iW \ne 0$ if and only if $\tau \le i \le \tau+d,$ \ \ \ $(0 \le i \le D)$.
\item[{\rm (iv)}] $W=\sum_{h=0}^d \, E_{\tau+h}W$ \ \ \ \mbox{\rm (orthogonal direct sum)}.
\end{itemize}
\end{lemma}
%--------------------------------------------------------------------------

Let $W$ and $W'$ denote irreducible $T$-modules. 
These $T$-modules are called {\de  isomorphic} whenever there exists 
a vector space  isomorphism $\sigma:W \to W^\prime$ such that 
$(\sigma S - S \sigma)W=0$  for all $S \in T$.
%$$
%(\sigma S - S \sigma)W=0 \ \ \ \mbox{for all } S \in T.
%$$
%The following is proven in  \cite[Lemma 4.3]{Talg2}.
%\begin{lemma}                       \label{lemma:isomorphicW}
%Let $W,W^\prime$ denote thin  irreducible $T$-modules that are isomorphic.
%Then they have the same endpoint, dual endpoint and diameter.
%\end{lemma}
Assume that the $T$-modules $W$ and $W'$ are isomorphic.
Then they have the same endpoint, dual endpoint and diameter. 
Now assume that the  $T$-modules $W$ and $W'$ are not isomorphic.
Then $W$ and $W'$ are orthogonal \cite[Lemma 3.3]{Curtin}.
%The following theorem is adapted from   \cite[Theorem 4.1]{Talg2}.
%
%\begin{theorem}
%Let $\G=(X,R)$ be distance regular graph with diameter $D \ge 3$.
%Suppose $\G$ is $Q$-polynomial with respect to the ordering
%$E_1,\dots,E_D$ of its primitive idempotents.
%
%Fix $x \in X$ and let $T=T(x)$ be the subconstituent algebra of $\G$
%with respect to  $x$.  Let $W$ denote  a thin, irreducible $T$-module
%with endpoint $\rho$, dual endpoint $\tau$ and diameter $d$. Then the following
%(i) and (ii) hold.
%\begin{itemize}
%\item[(i)] Pick a nonzero $u in E_\rho W$, and a nonzero $v \in E_\tau W$. Then
%           \begin{eqnarray*}
%           S &=& (E_\rho v,E_{\rho+1} v,\dots,E_{\rho+d} v),\\
%           S^* &=& (E_\rho^* v,E_{\rho+1}^* v,\dots,E_{\rho+d}^* v)
%           \end{eqnarray*}
%           are bases for $W$.
%\item[(ii)] The matrices $A$ and $A^*$ act as a Leonard pair on $W$.
%            In particular, with respect to the basis $S$, $A$ is diagonal
%            and $A^*$ is irreducible tridiagonal. With respect to the basis $S^*$, 
%            $A^*$ is diagonal  and $A$ is irreducible tridiagonal.
%\end{itemize}
%\end{theorem}
Let $\Psi$ denote the set of isomorphism classes of irreducible $T$-modules.
The elements of $\Psi$ are called {\de types}. For $\psi \in \Psi$
let $V_\psi$ denote the subspace of $V$ spanned by the irreducible
$T$-modules of type $\psi$. 
Call $V_\psi$ the {\de homogeneous component of $V$} of type $\psi$.
Observe that $V_\psi$ is a $T$-module. We have the
 decomposition
\begin{eqnarray}                                 \label{eq:decompose}
V = \sum_{\psi \in \Psi} V_\psi \ \ \ \ \mbox{(orthogonal direct  sum)}.
\end{eqnarray}
For $\psi \in \Psi$ define $e_\psi \in 
 \Mat_X(\C)$ such that
 $(e_\psi -I)V_\psi=0$ and 
 $e_\psi V_\lambda = 0 $ for all $\lambda \in \Psi$ with $\lambda \not=\psi$.
By construction
$I =\sum_{\psi \in \Psi} e_\psi$. Moreover
$e_\psi e_\lambda = \delta_{\psi,\lambda} e_\psi
$ for
$\psi, \lambda \in \Psi$.
According to the Wedderburn theory \cite[Chapter IV]{CurtisReiner}
the elements $\lbrace e_\psi \rbrace_{\psi \in \Psi}$ form
a basis for the center of $T$.

%--------------------------------------------------------------------------
%----------------------------- 5se-----------------------------------
\section{The main results}

We continue to discuss the subconstituent algebra
$T$ of our distance-regular graph $\Gamma$.
In this section we obtain our main results,
which concern how 
$T$ is related to the universal Askey-Wilson algebra
$\Delta_q$.
\medskip

Recall the $Q$-polynomial ordering $\{E_i\}_{i=1}^D$ of the nontrivial
primitive idempotents of $\G$.

%--------------------------------------------------------------------------
\begin{definition}                              
{\rm \cite[Section~1]{ItoTerwilliger1}.}
\label{def:QRacah}
Our  $Q$-polynomial structure $\{E_i\}_{i=1}^D$ is said to have
{\de $q$-Racah type} whenever there exist  scalars 
$u, u^*, v, v^*, w, w^*$ in $\C$ with 
$u, u^*, v, v^*$ nonzero such that both
$$
\theta_i = w + uq^{2i-D }+ vq^{D-2i}, 
\ \ \ \ \ \ \ \ \ \
\theta_i^* = w^* + u^*q^{2i-D }+ v^*q^{D-2i}
$$
for $0 \le i \le D$.
% Note that $q^{2i}  \ne  1$ for $1 \le i \le D$
\end{definition}
\begin{definition}    
The graph $\G$ is said to have {\de $q$-Racah type} whenever $\G$ has a
$Q$-polynomial structure  of $q$-Racah type.
\end{definition}
%--------------------------------------------------------------------------
\medskip

{F}rom now on, assume that
our $Q$-polynomial structure $\{E_i\}_{i=1}^D$
has $q$-Racah type, as in Definition \ref{def:QRacah}.
In Section \ref{sec:prelim} we defined the adjacency matrix $A$
and the dual adjacency matrix $A^*$.
We now adjust  $A$ and $A^*$.
Define  scalars  $a,b \in \C$ such that 
$a^2 = u/v$ and $b^2 = u^*/v^*$. Note that $a$ and $b$ are nonzero.
Define matrices $\cA, \cB$ in $\Mat_X(\C)$ as follows:
\begin{equation}                                   \label{eq:defcAcB}
\cA=\frac{A-wI}{av}, \ \ \ \ \ \ \ \ \ \
\cB=\frac{A^*-w^*I}{bv^*}.
%\cA=\frac{1}{\sqrt{uv}}(A-wI), \ \ \ \ \
%\cB=\frac{1}{\sqrt{u^*v^*}}(A-w^*I).
\end{equation}
%Note that the matrices $\cA$ and $\cB$ have the same primitive idempotents
%as $A$ and $A^*$, respectively.

%--------------------------------------------------------------------------
\begin{lemma}                              \label{lemma:tildeeigenvalue}
The following hold for $0 \le i \le D$.
\begin{itemize}
\item[{\rm (i)}] The  eigenvalue of  $\cA$ associated with
$E_i$ is
\begin{eqnarray}                    \label{eq:eigenAtilde}
aq^{2i-D }+ a^{-1}q^{D-2i}.
\end{eqnarray}
\item[{\rm (ii)}] The  eigenvalue of  $\cB$ associated with
$E_i^*$ is
\begin{eqnarray}
bq^{2i-D }+ b^{-1}q^{D-2i}.         \label{eq:eigenAtildestar}
\end{eqnarray}
\end{itemize}
\end{lemma}
%--------------------------------------------------------------------------

\begin{proof}
(i) By (\ref{prodAEi0}) and since
(\ref{eq:eigenAtilde}) is equal to $(\theta_i -w)/(av)$.\\
(ii) By (\ref{prodAstarEi0}) and since
(\ref{eq:eigenAtildestar}) is equal to $(\theta^*_i -w^*)/(bv^*)$.
% matrices $\cA$ and $\cB$ were obtained from
% $A$ and $A^*$ using affine transformations.
%Their  eigenvalues are obtained from the eigenvalues of $A$ and $A^*$, respectively,
%using the same affine transformations.
\end{proof}

We mentioned in Section \ref{sec:prelim} that $T$ is generated by $A,A^*$.
By this and (\ref{eq:defcAcB}) we see that $T$ is generated by $\cA,\cB$.

%--------------------------------------------------------------------------
\begin{lemma}                              \label{lemma:moduleigenvalue}
Let $W$ denote a thin irreducible $T$-module with 
endpoint $\rho$,  dual endpoint $\tau$, and diameter $d$. Then
$(\cA, \{E_i\}_{i=\tau}^{\tau + d}, \cB, \{E_i^*\}_{i=\rho}^{\rho+d})$ 
acts on $W$ as a Leonard system of $q$-Racah type.  
For this Leonard system  the eigenvalue sequence is
\begin{eqnarray}
a(W)q^{2i-d }+ a(W)^{-1}q^{d-2i}  \ \ \ \ \ \ \ \ (0 \le i \le d),
                                                \label{eq:WeigenAtilde}
\end{eqnarray}
where $a(W)=aq^{2\tau+d-D}$. Moreover  the dual eigenvalue sequence is
\begin{eqnarray}
b(W)q^{2i-d }+ b(W)^{-1}q^{d-2i}  \ \ \ \ \ \ \ \ (0 \le i \le d),
                                                \label{eq:WeigenAtildestar}
\end{eqnarray}
where $b(W)=bq^{2\rho+d-D}$.   
\end{lemma} 
%--------------------------------------------------------------------------

\begin{proof}
%Denote $\Phi=(\cA, \{E_i\}_{i=\tau}^{\tau + d}, \cB, \{E_i^*\}_{i=\rho}^{\rho+d})$. 
%By \cite[Theorem 4.1]{Talg2} and since the $T$-module $W$ is thin, 
%the sequence
%$(A, \{E_i\}_{i=\tau}^{\tau + d}, A^*, \{E_i^*\}_{i=\rho}^{\rho+d})$ 
%acts on $W$ as a Leonard system. 
By \cite[Theorem 4.1]{Talg2} and since %the $T$-module 
$W$ is thin, 
the sequence
$(A, \{E_i\}_{i=\tau}^{\tau + d}, A^*, \{E_i^*\}_{i=\rho}^{\rho+d})$ 
acts on $W$ as a Leonard system. 
By this and
(\ref{eq:defcAcB}), the sequence
$\Phi=(\cA, \{E_i\}_{i=\tau}^{\tau + d}, 
\cB, \{E_i^*\}_{i=\rho}^{\rho+d})$
acts on $W$ as a Leonard system.
Using
 Lemma  \ref{lemma:tildeeigenvalue}
     one checks that $\Phi$ has
eigenvalue sequence 
(\ref{eq:WeigenAtilde}) and
dual eigenvalue seqence 
(\ref{eq:WeigenAtildestar}).
Consequently 
$\Phi$ 
has $q$-Racah type.
\end{proof}

%%Let $W$ and $U$ denote nonorthogonal irreducible $T$-modules.
%%They are isomorphic as $T$-modules by \cite[Lemma 3.3]{Curtin}.
%Let $\Psi$ denote the set of isomorphism classes of irreducible $T$-modules.
%The elements of $\Psi$ are called {\de types}. For $\psi \in \Psi$
%we define $V_\psi$ to be the subspace of $V$ spanned by the irreducible
%$T$-modules of type $\psi$. Call $V_\psi$ the {\de homogeneous component of $V$}.
%Observe that $V_\psi$ is a $T$-module. Now  $V$ can 
%be decomposed as
%\begin{eqnarray}                                 \label{eq:decompose}
%V = \sum_{\psi \in \Psi} V_\psi \ \ \ \ \mbox{(orthogonal direct  sum)}.
%\end{eqnarray}
%%%%%%%%%%%

Given $\psi \in \Psi$, let $W$ denote an irreducible 
$T$-module of type
$\psi$. Then 
the  endpoint, dual endpoint, and diameter of $W$  are
independent of the choice of $W$, and 
%dimension,
depend only on $\psi$.   % by \cite[Lemma 4.3]{Talg2}.  % for thin
Assume that $W$ is thin.
Then the scalars
$a(W)$, $b(W)$ from 
Lemma \ref{lemma:moduleigenvalue}  are independent
of the choice of $W$, and
depend only on $\psi$. Sometimes we write
$a(W)=a(\psi)$ and
$b(W)=b(\psi)$.
Sometimes we write $d(\psi)$ for
the diameter of $W$.

%--------------------------------------------------------------------------
\begin{definition}
\label{def:boldabd}
\rm
Assume that each irreducible $T$-module is thin. Define
\begin{eqnarray*}
 \ba = \sum_{\psi \in \Psi} a(\psi)e_\psi,
 \qquad \qquad
 \bb = \sum_{\psi \in \Psi} b(\psi)e_\psi,
 \qquad \qquad 
 \Lambda = \sum_{\psi \in \Psi}   q^{d(\psi)+1} e_\psi.
 \end{eqnarray*}
\end{definition}
%--------------------------------------------------------------------------

%--------------------------------------------------------------------------
\begin{lemma} 
\label{def:boldabdinv}
Assume that each irreducible $T$-module is thin.
Then the matrices
$\ba$, $\bb$, $\Lambda $ from Lemma
{\rm \ref{def:boldabd}} are invertible.
Their inverses are
\begin{eqnarray*}
 \ba^{-1} = \sum_{\psi \in \Psi} a(\psi)^{-1}e_\psi,
 \quad \qquad
 \bb^{-1} = \sum_{\psi \in \Psi} b(\psi)^{-1}e_\psi,
 \quad \qquad 
 \Lambda^{-1} = \sum_{\psi \in \Psi}   q^{-d(\psi)-1} e_\psi.
 \end{eqnarray*}
\end{lemma}
%--------------------------------------------------------------------------
\begin{proof} By Definition \ref{def:boldabd} and the discussion
below (\ref{eq:decompose}).
\end{proof}

%--------------------------------------------------------------------------
\begin{definition}                         \label{def:cw}
Let $W$ denote a thin irreducible $T$-module. 
By Lemma \ref{lemma:moduleigenvalue} the sequence 
$(\cA, \{E_i\}_{i=\tau}^{\tau + d},$ $\cB, \{E_i^*\}_{i=\rho}^{\rho+d})$
acts on $W$ as a Leonard system of $q$-Racah type.
Let $c=c(W)$ denote the corresponding scalar from Section \ref{sec:LP}.
The scalar $c$ is defined up to reciprocal.
We sometimes write $c=c(\psi)$, where $\psi$ denotes the type of $W$.
\end{definition}
%--------------------------------------------------------------------------

%--------------------------------------------------------------------------
\begin{definition} 
\label{def:boldc}
\rm Assume that each irreducible $T$-module is thin.
Define
\begin{eqnarray*}
\bc = \sum_{\psi \in \Psi} c(\psi)e_\psi.
\end{eqnarray*}
\end{definition}
%--------------------------------------------------------------------------

%--------------------------------------------------------------------------
\begin{lemma}       \label{lemma:cinv}
Assume that each irreducible $T$-module is thin.
Then the matrix $\bc$ from Definition {\rm \ref{def:boldc}} is invertible. 
The inverse is
\begin{eqnarray*}
\bc^{-1} = \sum_{\psi \in \Psi} c(\psi)^{-1}e_\psi.
\end{eqnarray*}
\end{lemma}
\begin{proof}
By Definition \ref{def:boldc} and the discussion below (\ref{eq:decompose}).
\end{proof}
%--------------------------------------------------------------------------

%--------------------------------------------------------------------------
\begin{lemma}                              \label{lemma:abc}
Assume that each irreducible $T$-module is thin. 
Recall the matrices $\ba,\bb,\Lambda$ from Definition {\rm \ref{def:boldabd}}
and the matrix $\bc$ from Definition {\rm \ref{def:boldc}}.
Then 
$\ba^{\pm 1}$, $\bb^{\pm 1}$, $\bc^{\pm 1}$, $\Lambda^{\pm 1}$ 
act on each irreducible $T$-module $W$ as 
$$a(\psi)^{\pm 1}I, \ \ \ \ b(\psi)^{\pm 1}I, \ \ \ \ c(\psi)^{\pm 1}I,
\ \ \ \  q^{\pm d(\psi)\pm1}I,$$
respectively, where $\psi$ is the type of $W$.
\end{lemma}
%--------------------------------------------------------------------------

\begin{proof}
Follows from the discussion below (\ref{eq:decompose}).
\end{proof}

%--------------------------------------------------------------------------
\begin{lemma}                              \label{lemma:abccentral}
Each of the matrices $\ba^{\pm 1}$, $\bb^{\pm 1}$, $\bc^{\pm 1}$,
$\Lambda^{\pm 1}$ is a central element of $T$.
\end{lemma}
%--------------------------------------------------------------------------

\begin{proof}
By our discussion at the end of Section 4,
each of
$\lbrace e_\psi\rbrace_{\psi \in \Psi}$ is a central
element of $T$. The result follows in view of
 Definition
\ref{def:boldabd} and
 Definition
\ref{def:boldc}.
\end{proof}

%--------------------------------------------------------------------------
\begin{theorem}                              \label{defineC}
Assume that each irreducible $T$-module is thin. Then there exists a unique 
$\cC \in T$ such that
\begin{eqnarray}
\cA + \frac{q\;\cB \cC-q^{-1}\cC \cB}{q^2-q^{-2}} &=&  
\frac{(\ba+\ba^{-1})(\Lambda+\Lambda^{-1})+(\bb+\bb^{-1})(\bc+\bc^{-1})}{q+q^{-1}},
       \label{eq:AWDRG1}\\
\cB + \frac{q\;\cC \cA -q^{-1}\cA \cC}{q^2-q^{-2}} &=& 
\frac{(\bb+\bb^{-1})(\Lambda+\Lambda^{-1})+(\bc+\bc^{-1})(\ba+\ba^{-1})}{q+q^{-1}},
       \label{eq:AWDRG2}\\
\cC +\frac{q\;\cA\cB-q^{-1}\cB\cA}{q^2-q^{-2}} &=& 
\frac{(\bc+\bc^{-1})(\Lambda+\Lambda^{-1})+(\ba+\ba^{-1})(\bb+\bb^{-1})}{q+q^{-1}}. 
       \label{eq:AWDRG3}
\end{eqnarray}
%Here $\cA,\cB$ are from {\rm (\ref{eq:defcAcB})} and 
%$\ba,\bb,\bc,\Lambda$ are from Lemma \ref{lemma:abc}.
\end{theorem}
%--------------------------------------------------------------------------

\begin{proof} Define $\cC$ such that
       (\ref{eq:AWDRG3}) holds.
Then $\cC \in T$ since $T$ contains
$\cA$, $\cB$  as well as
$\ba^{\pm 1}$, $\bb^{\pm 1}$, $\bc^{\pm 1}$, $\Lambda^{\pm 1}$.
We now verify 
       (\ref{eq:AWDRG1}), 
       (\ref{eq:AWDRG2}).
To do this, it suffices to show that
       (\ref{eq:AWDRG1}), 
       (\ref{eq:AWDRG2}) hold on each 
irreducible $T$-module. But this follows from
the construction
and
Theorem  \ref{thm:AskeyWilsonZ3}.
%Decompose $V$ into an orthogonal direct sum of irreducible $T$-modules,
%\begin{equation}                        \label{eq:sum}
%V=\sum_{i=1}^{m}\, W_i.
%\end{equation}
%Let $W$ denote an irreducible $T$-module from that sum with 
%endpoint $\rho$,  dual endpoint $\tau$ and diameter $d$. 
%By assumption, $W$ is thin. Therefore           
%$(\cA, \{E_i\}_{i=\tau}^{\tau + d}, \cB, \{E_i^*\}_{i=\rho}^{\rho+d})$ 
%acts on $W$ as a Leonard system of $q$-Racah type 
%by Lemma \ref{lemma:moduleigenvalue}. 
%Let $a(W),b(W)$ be from Lemma \ref{lemma:moduleigenvalue} and
%let   $c(W)$ be from Definition \ref{def:cw}. 
%Then  there exists a unique $A^\eps(W)$
%such that for $\cA,\cB,A^\eps(W)$ the Askey-Wilson
%relations (\ref{eq:AWsim1})--(\ref{eq:AWsim3}) hold on $W$ by 
%Lemma \ref{thm:AskeyWilsonZ3}.
%
%Define $\cC$ to be the unique matrix that acts on $W_i$ as $A^\eps(W_i)$
%for $1 \le i \le m$.
%Obviously (\ref{eq:AWDRG1})--(\ref{eq:AWDRG3}) are satisfied for such $\cC$ since 
%they are satisfied on each of the irreducible $T$-modules from the sum (\ref{eq:sum}).
%Also,  $\cC$ is an element of $T$ since it is a polynomial in
%elements of $T$ by (\ref{eq:AWDRG3}) and Lemma \ref{lemma:abccentral}. 
\end{proof}

%--------------------------------------------------------------------------
\begin{theorem}                              \label{central}
Assume that each irreducible $T$-module is thin.
Then each of 
\begin{eqnarray}               \label{eq:centralABC}
\cA + \frac{q\;\cB \cC-q^{-1}\cC \cB}{q^2-q^{-2}}, \ \ \ \
\cB + \frac{q\;\cC \cA -q^{-1}\cA \cC}{q^2-q^{-2}}, \ \ \ \
\cC + \frac{q\;\cA\cB-q^{-1}\cB\cA}{q^2-q^{-2}}
\end{eqnarray}
is central in $T$. The matrices  $\cA,\cB$ are from {\rm (\ref{eq:defcAcB})} 
and $\cC$ is from Theorem {\rm \ref{defineC}}.
\end{theorem}
%--------------------------------------------------------------------------

\begin{proof}
In 
each equation (\ref{eq:AWDRG1})--(\ref{eq:AWDRG3}), 
the matrix on the right is a central element of $T$
by 
Lemma \ref{lemma:abccentral}. 
\end{proof}

%--------------------------------------------------------------------------
\begin{theorem}
\label{thm:hom}
Assume that each irreducible $T$-module is thin. 
Then there exists a surjective $\C$-algebra homomorphism 
$\Delta_q \to T$ that sends
$$
A \mapsto \cA, \quad \ \ \ \ B \mapsto \cB, \quad \ \  \ \ C \mapsto \cC.
$$
The matrices  $\cA,\cB$ are from {\rm (\ref{eq:defcAcB})} 
and $\cC$ is from Theorem {\rm \ref{defineC}}.
\end{theorem}
%--------------------------------------------------------------------------

\begin{proof}
By Theorem \ref{central} and since $T$ is generated by  $\cA$, $\cB$. 
\end{proof}

%--------------------------------------------------------------------------
\begin{theorem}                    \label{thm:moduleaction}        
Assume that each irreducible $T$-module is thin. 
Then the standard module $V$ becomes a $\Delta_q$-module on which
the $\Delta_q$-generators $A,B,C$ act as follows:
\medskip

\begin{center}
\begin{tabular}{l|ccc}
$\Delta_q$-generator & $A$ & $B$ & $C$ \\
\hline
action on $V$ & $\cA$ & $\cB$ & $\cC$
\end{tabular}
\end{center}
\medskip

\noindent
The matrices  $\cA,\cB$ are from {\rm (\ref{eq:defcAcB})} 
and $\cC$ is from Theorem {\rm \ref{defineC}}.
\end{theorem}
%--------------------------------------------------------------------------

\begin{proof}
Pull back the $T$-module structure on $V$, using
the homomorphism $\Delta_q \to T$ from Theorem
\ref{thm:hom}.
\end{proof}

%--------------------------------------------------------------------------
%----------------------------- 6se-----------------------------------
\section{2-Homogeneous bipartite distance-regular graphs}

In this section we  work out an example.
Throughout this section we assume that $\G$ is bipartite and 
$2$-homogeneous in the sense of Nomura \cite{Nomura},
but not a hypercube.
For a thorough description of such $\G$, see \cite{Curtin1}, \cite{Curtin2}.
By \cite[Theorem 42]{Curtin1} $\G$ is $Q$-polynomial.
Let $\{E_i\}_{i=1}^D$ denote a $Q$-polynomial ordering 
of the nontrivial primitive idempotents of $\G$. 
Let $\{\theta_i\}_{i=0}^D$ and $\{\theta_i^*\}_{i=0}^D$ denote
the corresponding eigenvalue and dual eigenvalue sequences.
The following lemma summarizes the properties of $\G$ that we need.

%--------------------------------------------------------------------------
\begin{lemma}{\rm \cite[Corollary~43]{Curtin1}.}    \label{thm:2homproperties}
The following {\rm (i), (ii)} hold.
\begin{itemize}
\item[{\rm (i)}] There exists a nonzero  $q \in \C$ such that $q^4 \ne  1$ and
\begin{equation}                   \label{eq_2hom_eigen}
\theta_i=\frac{q^{D-2}+q^{2-D}}{q^2-q^{-2}}(q^{D-2i}-q^{2i-D})
\ \ \ \ \ \ \ \ \ \ (0 \le i \le D).
\end{equation}
\item[{\rm (ii)}] $\theta_i^*=\theta_i  \ \ \ \   (0 \le i \le D).$
\end{itemize}
\end{lemma}
%--------------------------------------------------------------------------

%--------------------------------------------------------------------------
\begin{corollary}                          \label{thm:2hom_qracah}
The above  $Q$-polynomial structure has $q$-Racah type,
      with $w=w^*=0$ and 
      \begin{equation}             \label{eq:calculateuv}
      v=v^*=-u=-u^*=\frac{q^{D-2}+q^{2-D}}{q^2-q^{-2}}.
      \end{equation}
\end{corollary}
%--------------------------------------------------------------------------

%--------------------------------------------------------------------------
\begin{note}
The scalar $q$ in \cite{Curtin2} corresponds to our $q^2$.
\end{note}
%--------------------------------------------------------------------------

Recall the subconstituent algebra $T=T(x)$ from Section 3.
The irreducible $T$-modules are described
in the following lemma.

%--------------------------------------------------------------------------
\begin{lemma}{\rm \cite[Theorem 4.1, Lemma 4.2]{Curtin2}.}    \label{thm_2hompmodules}
Let $W$ denote an irreducible $T$-module, with endpoint  $\rho$. 
Then the following {\rm (i)--(iv)} hold.
\begin{itemize}
\item[{\rm (i)}] $0 \le \rho \le \lfloor D/2 \rfloor$.
\item[{\rm (ii)}] Up to isomorphism, $W$ is the unique irreducible
                   $T$-module with endpoint $\rho$.
\item[{\rm (iii)}] $W$ is thin.    
\item[{\rm (iv)}] $W$ has dual endpoint $\rho$ and diameter $D-2\rho$.
\end{itemize}
%Moreover, for all $\rho$, $(0 \le \rho \le \lfloor D/2 \rfloor)$,
%there exists an irreducible $T$-module with endpoint $\rho$.
\end{lemma}
%--------------------------------------------------------------------------

In Section 5 we displayed a $\Delta_q$ action on the standard module of $T$.
In this section we describe the $\Delta_q$ action in more detail.
Fix  $\bi \in \C$ such that $\bi^2=-1$.

%--------------------------------------------------------------------------
\begin{lemma}              \label{step1}
Recall the scalars $a,b$ from above Lemma {\rm \ref{lemma:tildeeigenvalue}}. 
Then $a,b \in \{\bi,-\bi\}$. Moreover,
\begin{equation}             \label{eq:2homAB}
\cA= A \, \frac{q^2-q^{-2}}{a(q^{D-2}+q^{2-D})}, \ \ \ \ \ \ \ \ \
\cB=  A^* \, \frac{q^2-q^{-2}}{b(q^{D-2}+q^{2-D})}.
\end{equation}
\end{lemma}

\begin{proof}
By (\ref{eq:calculateuv}) and the construction
$a^2=u/v=-1$, so  $a \in \{\bi,-\bi\}$.
Similarly $b \in \{\bi,-\bi\}$.
%By these comments and (\ref{eq:defcAcB})
By Corollary \ref{thm:2hom_qracah} and (\ref{eq:defcAcB})
we obtain (\ref{eq:2homAB}).
\end{proof}

%--------------------------------------------------------------------------
\begin{lemma}              \label{step2}
Let $W$ denote an irreducible $T$-module.
Then 
\begin{equation}                 \label{eq:calculateabW}
a(W)=a, \ \ \ \ \ \ \ \ \ b(W)=b.
\end{equation}
Moreover 
\begin{equation}                 \label{eq:calculatecW}
c(W)\in \{\bi,-\bi\}.
\end{equation}
\end{lemma}

\begin{proof}
Let   $\rho$ denote the endpoint of $W$.
By Lemma \ref{thm_2hompmodules}, $W$ has 
dual endpoint $\rho$ and diameter $d=D-2\rho$. Moreover $W$ is thin.
By Lemma \ref{lemma:moduleigenvalue} we obtain $a(W)=a$ and $b(W)=b$.
We now show 
(\ref{eq:calculatecW}). 
First assume that $d=0$. Setting
$\kappa=0$ in
(\ref{eq:kappa}) we find $c(W)^2=-1$, so
(\ref{eq:calculatecW}) holds. Next assume that $d\geq 1$.
By Lemma \ref{lemma:moduleigenvalue}, the sequence 
$(\cA, \{E_i\}_{i=\rho}^{\rho + d},\cB, \{E_i^*\}_{i=\rho}^{\rho+d})$ 
acts on $W$ as a Leonard system, which we denote by $\Phi$. 
To find $c(W)$ we apply (\ref{eq:kappa}), (\ref{eq:calculatekappa}) to $\Phi$. 
We first obtain $\phi_1(\Phi)$ using (\ref{eq:calculatphi}).  
By (\ref{eq:calculatphi}) and since $\G$ is bipartite,
\begin{eqnarray*}
\phi_1(\Phi)&=&-(\theta^*_0(\Phi)-\theta^*_1(\Phi))\theta_d(\Phi)\\
%         &=&(\bi q^{2\rho-D} + \bi^{-1} q^{D-2\rho})
%            (\bi q^{2\rho+2-D}+ \bi^{-1} q^{D-2\rho-2} - \bi q^{2\rho-D}- \bi^{-1} q^{D-2\rho})
%         &=&-(\theta^*_\rho-\theta^*_{\rho+1})\theta_{\rho+d}\\
%         &=&ab(q-q^{-1})(q^{2\rho+1-D}+q^{D-2\rho-1})(q^{2\rho-D} - q^{D-2\rho}).
   &=& -(bq^{-d}+b^{-1}q^d-bq^{2-d}-b^{-1}q^{d-2})(aq^d+a^{-1}q^{-d})\\
   &=& ab(q-q^{-1})(q^{1-d}+q^{d-1})(q^d-q^{-d}) .
\end{eqnarray*}
By this and (\ref{eq:calculatekappa}) we obtain $\kappa(\Phi)=0$. 
Therefore $c(W)^2=-1$ in view of (\ref{eq:kappa}),
and (\ref{eq:calculatecW}) follows.
\end{proof}

%--------------------------------------------------------------------------
\begin{lemma}              \label{step3}
We have $\ba= a I$ and $\bb= b I$. 
\end{lemma}

\begin{proof}
We show that $\ba = a I$. 
The element $\ba$ is given in Definition \ref{def:boldabd}.
By (\ref{eq:calculateabW}), $a(\psi)=a$ for all $\psi \in \Psi$.
We mentioned at the end of Section 4 that 
$I=\sum_{\psi \in \Psi} \, e_\psi$. By these comments $\ba = a I$. 
One similarly shows that $\bb = b I$.
\end{proof}

%--------------------------------------------------------------------------
\begin{lemma}              \label{step4}
We have 
\begin{equation}
\ba+\ba^{-1}= 0, \ \ \ \ \ \
\bb+\bb^{-1}= 0, \ \ \ \ \ \
\bc+\bc^{-1}= 0. 
\end{equation}
\end{lemma}
\begin{proof}
The first two equations follow from Lemma \ref{step3} and $a,b \in \{\bi,-\bi\}$.
The third equation follows from Definition \ref{def:boldc},
Lemma \ref{lemma:cinv}, and (\ref{eq:calculatecW}).
\end{proof}

%--------------------------------------------------------------------------
\begin{theorem}              \label{step5}
Let $\cC$ be as in Theorem {\rm \ref{defineC}}. Then 
\begin{eqnarray}
\cA + \frac{q\;\cB \cC-q^{-1}\cC \cB}{q^2-q^{-2}} &=&  0,  \label{eq:AW1a}\\
\cB + \frac{q\;\cC \cA -q^{-1}\cA \cC}{q^2-q^{-2}} &=&  0, \label{eq:AW2a}\\
\cC +\frac{q\;\cA\cB-q^{-1}\cB\cA}{q^2-q^{-2}} &=& 0. \label{eq:AW3a}
\end{eqnarray}
\end{theorem}

\begin{proof}
Evaluate (\ref{eq:AWDRG1})--(\ref{eq:AWDRG3}) using Lemma \ref{step4}.
\end{proof}

%--------------------------------------------------------------------------
\begin{rem}
The equations (\ref{eq:AW1a})--(\ref{eq:AW3a}) can be obtained 
directly from \cite[Lemma 3.3]{Curtin2}. 
\end{rem}
%--------------------------------------------------------------------------

%%----------------------------- 7se-----------------------------------
%\section{Conclusion}
%
%\texttt{TODO ???\\
%-- combinatorial meaning of $\cC$\\
%-- are there any relations among $\Phi,\Phi^*,\Phi^\eps$?\\ 
%-- is $\cC$ the imaginary adjacency matrix?
%}

%-------------------------- bib ---------------------------------------

\bigskip

\noindent
Paul Terwilliger \\
Department of Mathematics \\
University of Wisconsin \\
480 Lincoln Drive \\
Madison, WI 53706-1388 USA \\
email: \texttt{terwilli@math.wisc.edu}

\bigskip

\noindent
Arjana \v Zitnik,\\
Faculty of Mathematics and Physics, University of Ljubljana and \\
Institute of Mathematics, Physics and Mechanics,\\
Jadranska 19, 1000 Ljubljana, Slovenia\\
email: \texttt{arjana.zitnik@fmf.uni-lj.si}

\end{document}